\documentclass[submission]{FPSAC2019}
\newtheorem{Theorem}{Theorem}[section]
\newtheorem{Lemma}[Theorem]{Lemma}
\newtheorem{definition}[Theorem]{Definition}
\newtheorem{Proposition}[Theorem]{Proposition}
\newtheorem{Corollary}[Theorem]{Corollary}

\definecolor{ColBlack}{RGB}{0,0,0} % Black.
\definecolor{ColWhite}{RGB}{255,255,255} % White.
\definecolor{Col1}{RGB}{133,6,6} % Rouge sang.
\definecolor{Col2}{RGB}{198,8,0} % Rouge ponceau.
\definecolor{Col3}{RGB}{174,74,52} % Rouge tomette.
\definecolor{Col4}{RGB}{103,113,121} % Gris de Payne.
\definecolor{Col5}{RGB}{90,94,107} % Ardoise.
\definecolor{Col6}{RGB}{70,63,50} % Taupe.
\definecolor{Col7}{RGB}{0,0,255} % Bleu roi.
\definecolor{Col8}{RGB}{24,74,217} % Blue
\definecolor{Col9}{RGB}{25,170,229} % Cyan

\newcommand{\Hide}[1]{\textcolor{Col4}{\tt [hidden]}}

\newcommand{\Def}[1]{\textcolor{ColBlack}{\em #1}}

\DeclareRobustCommand{\gobblefive}[5]{}

\let\SavedCaption=\caption
\renewcommand*{\caption}[2][\shortcaption]{%
    \def\shortcaption{#2}
    \SavedCaption[\; #1]{#2}}

\let\SavedParagraph=\paragraph
\renewcommand{\paragraph}[1]{%
    \SavedParagraph{\it #1}}

\newcommand*\ClearToLeftPage{%
    \clearpage
    \ifodd\value{page}
        \hbox{}
        \vspace*{\fill}
        \thispagestyle{empty}
        \newpage
    \fi
}

\usepackage[utf8x]{inputenc}
\usepackage[french, english]{babel}
\usepackage{amsmath,amsfonts,amssymb,amsthm,bm,shuffle}
\usepackage{pgf,tikz}
\usetikzlibrary{arrows}
\usepackage{float} 

\usepackage{mathrsfs}
\usepackage{amsfonts}
\usepackage{pb-diagram}

\usepackage{tikz}
\usepackage{tikz-cd} 
\usetikzlibrary{arrows, matrix}

\usepackage{xcolor}%\usepackage[hyperindex=true,frenchlinks=true,colorlinks=true,citecolor=Col8,linkcolor=Col7,urlcolor=Col9,linktocpage,pagebackref=true]{hyperref}

\usepackage{tikz}
\usetikzlibrary{shapes}
\usetikzlibrary{fit}
\usetikzlibrary{decorations.pathmorphing}
\usetikzlibrary{calc}

\usepackage{mathtools}
\usepackage{dsfont}
\usepackage{wasysym}
\usepackage{stmaryrd}
\usepackage{nameref}
\usepackage{youngtab}
\usepackage{cite}
\usepackage{xspace}
\usepackage{subfig}
\usepackage{multirow}
\usepackage{enumitem}
\usepackage{multicol}
\usepackage{chngcntr}

\numberwithin{equation}{section} %{subsection}

\counterwithin{figure}{section}
\counterwithin{table}{section}

\renewcommand{\leq}{\leqslant}
\renewcommand{\geq}{\geqslant}

\newcommand{\ip}{\mathcal{IP}_{n}}
\newcommand{\cc}{\mathcal{CC}_{n}}
\newcommand{\dit}{\mathcal{TID}_{n}}

\newcommand{\itam}{\mathcal{TI}_{n}}
\newcommand{\ccs}{\mathcal{CC}^{sync}_{n}}

\usepackage{xcolor}
\title{Cubic realizations of Tamari interval lattices}

\author{Camille Combe\addressmark{1}}

\address{\addressmark{1}Institut de Recherche Mathématique Avancée 
UMR 7501, Université de Strasbourg et CNRS, France.}

\received{\today}

\abstract{We introduce cubic coordinates, which are integer words encoding intervals in the Tamari lattices.
Cubic coordinates are in bijection with interval-posets, themselves known to be in bijection with Tamari intervals.
We show that in each degree the set of cubic coordinates forms a lattice, isomorphic to the lattice of Tamari intervals. Geometric realizations are naturally obtained by placing cubic coordinates in space, highlighting some of their properties. Finally, we consider the cellular structure of these realizations.}

\keywords{Tamari lattices, Tamari intervals, 
interval-posets, posets, geometric realizations, cubical complexes.}

\begin{document}

\maketitle
%% note that you DO NOT have to put your abstract here -- it is generated by \maketitle and the \abstract and \resume commands above

%%%%%%%%%%%%%%%%%%%%%%%%%%%%%%%%%%%%%%%%%%%%%%%%
%%%%%%%%%%%%%%%%%%%%%%%%%%%%%%%%%%%%%%%%%%%%%%%%
\section*{Introduction}

The Tamari lattices are partial orders having extremely rich combinatorial and algebraic properties. These partial orders are defined on the set of binary trees and rely on the rotation operation~\cite{Tam62}.
We are interested in the intervals of these lattices, meaning the pairs of comparable binary trees. Tamari intervals of size $n$ also form a lattice. The number of these objects is given by a formula that was proved by Chapoton~\cite{Cha06}:
\begin{equation*}
\frac{2(4n+1)!}{(n+1)!(3n+2)!}.
\end{equation*}
Strongly linked with associahedra, Tamari lattices have been recently generalized in many ways~\cite{BPR12,PRV17}. In this process, the number of intervals of these generalized lattices have also been enumerated through beautiful formulas~\cite{BMFPR12,FPR17}.
Many bijections between Tamari intervals of size $n$ and other combinatorial objects are known. For instance, a bijection with planar triangulations is presented by Bernardi and Bonichon in~\cite{BB09}. It has been proved by Châtel and Pons that Tamari intervals are in bijection with interval-posets of the same size~\cite{CP15}.

We provide in this paper a new bijection with Tamari intervals, which is inspired by interval-posets. 
More precisely, we first build two words of size $n$ from the Tamari diagrams~\cite{Pal86} of a binary tree. Then, if they satisfy a certain property of compatibility, we build a Tamari interval diagram from these two words. 
We show that Tamari interval diagrams and interval-posets are in bijection.
Then, we propose a new encoding of Tamari intervals, by building $(n-1)$-tuples of numbers from Tamari interval diagrams. These tuples we refer to as cubic coordinates. 
This new coding has two obvious virtues: it is very 
compact and it 
gives a way of comparing in a simple manner two Tamari intervals, through a fast algorithm. 
On the other hand, some properties of Tamari intervals translate nicely in the setting of cubic coordinates. For instance, synchronized Tamari intervals~\cite{PRV17} become cubic coordinates with no zero entry.
Besides, cubic coordinates provide naturally a geometric realization of the lattice of Tamari intervals, by seeing them as space coordinates. Indeed, all cubic coordinates of size $n$ can be placed in the space $\mathbb{R}^{n-1}$. By drawing their cover relations, we obtain an oriented graph. This gives us a realization of cubic coordinate lattices, which we call cubic realization. This realization leads us to many questions, in particular about the cells it contains.

This paper is organized as follows.
In Section~\ref{Part2}, we define Tamari interval diagrams from Tamari diagrams. Then we give a bijection between the set of Tamari interval diagrams and the set of interval-posets. 
Subsequently, we define in Section~\ref{Part3} cubic coordinates, and build a bijection between these and Tamari interval diagrams.
By using these two bijections, and after endowing the set of cubic coordinates with a partial order structure, we show in Section~\ref{Part4} that there is an isomorphism of posets between cubic coordinate posets and Tamari interval lattices. 
Finally, the cubic realization and the cells it contains make the subject of Section~\ref{Part5}. We also show here how to associate a synchronized cubic coordinate with each cell.

Some proofs of the presented results are omitted and
some others are sketched in this extended abstract.

{\bf Notations.} Throughout this article, for all words $u$, we denote by $u_i$ the $i$-th letter of $u$. We use the notation $[n]$ to denote the set $\{1, \dots, n\}$.

%%%%%%%%%%%%%%%%%%%%%%%%%%%%%%%%%%%%%%%%%%%%%%%%
%%%%%%%%%%%%%%%%%%%%%%%%%%%%%%%%%%%%%%%%%%%%%%%%
\section{Tamari interval diagrams}\label{Part2}

In this section, we recall the definition of Tamari diagrams \cite{Pal86} and generalize this notion in order to define Tamari interval diagrams. Then, we establish a bijection between the set of Tamari interval diagrams and the set of interval-posets.
\smallbreak

\begin{definition}\label{diagTam}
A \Def{Tamari diagram} is a word $u=u_1u_2\dots u_n$ of integers
such that
\begin{enumerate}[label=(\roman*)]
\item $0\leq u_i \leq n - i$ for all $i\in [n]$;
\item  \label{condDT2} $u_{i+j} \leq u_i - j$ for all $i\in [n]$ and $0\leq j\leq u_i$.
\end{enumerate}
The size of a Tamari diagram is its number of letters.
\end{definition}
For instance, the fourteen Tamari diagrams of size $4$
are
$$
0000, 0010, 0100, 0200, 0210, 1000, 1010, 
2000, 2100, 3000, 3010, 3100, 3200, 3210.
$$
The set of Tamari diagrams of size $n$ is in bijection with the set of binary trees of the same size~\cite{Pal86}.
In the sequel, we need to encode a pair of binary trees with $n$ nodes by two words of size $n$. To this aim,
we introduce here dual Tamari diagrams. The first binary 
tree of the pair is encoded by its Tamari diagram and the
second is encoded by its dual Tamari diagram.

\begin{definition}\label{diagTamDual}
A \Def{dual Tamari diagram} is a word $v=v_1v_2\dots v_n$ of integers such that
\begin{enumerate}[label=(\roman*)]
\item \label{condDTD1} $0\leq v_i \leq i - 1$ for all $i\in [n]$;
\item \label{condDTD2} $v_{i-j} \leq v_i - j$ for all $i\in [n]$ and $0\leq j\leq v_i$.
\end{enumerate}
The size of a dual Tamari diagram is its number of letters.
\end{definition}

A word $v = v_1v_2\dots v_n$ is a dual Tamari diagram if
and only if its reversal is a Tamari diagram.
We can use for Tamari diagrams and dual Tamari diagrams the graphical representation proposed in~\cite{Gir11}. The value of a letter of both diagrams gives the height of the corresponding column. Condition~\ref{condDT2} of Definition~\ref{diagTam} (resp. \ref{condDTD2} of Definition~\ref{diagTamDual}) translates in the following way on the drawing: from each column, one draws a dotted line of slope $-1$ (resp. $1$), and the column
to its right (resp. its left) 
must not cross this line.
An example is given in Figure~\ref{dtdtd}.

\definecolor{wrwrwr}{rgb}{0.3803921568627451,0.3803921568627451,0.3803921568627451}\definecolor{rvwvcq}{rgb}{0.08235294117647059,0.396078431372549,0.7529411764705882}\definecolor{cqcqcq}{rgb}{0.7529411764705882,0.7529411764705882,0.7529411764705882}
\definecolor{ffqqqq}{rgb}{1,0,0}\definecolor{qqqqff}{rgb}{0,0,1}

\begin{figure}[h!]
\centerline{\scalebox{0.7}{\begin{tikzpicture}[line cap=round,line join=round,>=triangle 45,x=1cm,y=1cm]\clip(-7.8,-0.4) rectangle (6.,5.);
\draw [line width=1.2pt] (0,0)-- (5,0);
\draw [line width=1.2pt,color=ffqqqq] (4.5,0)-- (4.5,0.5);
\draw [line width=1.2pt,color=ffqqqq] (4.5,0.5)-- (5,0.5);
\draw [line width=1.2pt,color=ffqqqq] (5,0.5)-- (5,0);
\draw [line width=1.2pt,color=ffqqqq] (3,0)-- (3,0.5);
\draw [line width=1.2pt,color=ffqqqq] (3,0.5)-- (3,1);
\draw [line width=1.2pt,color=ffqqqq] (3,1)-- (3,1.5);
\draw [line width=1.2pt,color=ffqqqq] (3,1.5)-- (3,2);
\draw [line width=1.2pt,color=ffqqqq] (3,2)-- (2.5,2);
\draw [line width=1.2pt,color=ffqqqq] (2.5,2)-- (2.5,1.5);
\draw [line width=1.2pt,color=ffqqqq] (2.5,1.5)-- (2.5,1);
\draw [line width=1.2pt,color=ffqqqq] (2.5,1)-- (2.5,0.5);
\draw [line width=1.2pt,color=ffqqqq] (2.5,0.5)-- (2.5,0);
\draw [line width=1.2pt,color=ffqqqq] (2.5,0.5)-- (3,0.5);
\draw [line width=1.2pt,color=ffqqqq] (2.5,1)-- (3,1);
\draw [line width=1.2pt,color=ffqqqq] (2.5,1.5)-- (3,1.5);
\draw [line width=1.2pt,color=ffqqqq] (1,0)-- (1,0.5);
\draw [line width=1.2pt,color=ffqqqq] (1,0.5)-- (1.5,0.5);
\draw [line width=1.2pt,color=ffqqqq] (1.5,0.5)-- (1.5,0);
\draw [line width=1.2pt] (-7,0)-- (-2,0);
\draw [line width=1.2pt,color=qqqqff] (-7,0)-- (-7,0.5);
\draw [line width=1.2pt,color=qqqqff] (-7,0.5)-- (-7,1);
\draw [line width=1.2pt,color=qqqqff] (-7,1)-- (-7,1.5);
\draw [line width=1.2pt,color=qqqqff] (-7,1.5)-- (-7,2);\draw [line width=1.2pt,color=qqqqff] (-7,2)-- (-7,2.5);\draw [line width=1.2pt,color=qqqqff] (-7,2.5)-- (-7,3);\draw [line width=1.2pt,color=qqqqff] (-7,3)-- (-7,3.5);\draw [line width=1.2pt,color=qqqqff] (-7,3.5)-- (-7,4);\draw [line width=1.2pt,color=qqqqff] (-7,4)-- (-7,4.5);\draw [line width=1.2pt,color=qqqqff] (-7,4.5)-- (-6.5,4.5);\draw [line width=1.2pt,color=qqqqff] (-6.5,4.5)-- (-6.5,4);\draw [line width=1.2pt,color=qqqqff] (-6.5,4)-- (-6.5,3.5);\draw [line width=1.2pt,color=qqqqff] (-6.5,3.5)-- (-6.5,3);\draw [line width=1.2pt,color=qqqqff] (-6.5,3)-- (-6.5,2.5);\draw [line width=1.2pt,color=qqqqff] (-6.5,2.5)-- (-6.5,2);\draw [line width=1.2pt,color=qqqqff] (-6.5,2)-- (-6.5,1.5);\draw [line width=1.2pt,color=qqqqff] (-6.5,1.5)-- (-6.5,1);\draw [line width=1.2pt,color=qqqqff] (-6.5,1)-- (-6.5,0.5);\draw [line width=1.2pt,color=qqqqff] (-6.5,0.5)-- (-6.5,0);\draw [line width=1.2pt,color=qqqqff] (-7,0.5)-- (-6.5,0.5);\draw [line width=1.2pt,color=qqqqff] (-7,1)-- (-6.5,1);\draw [line width=1.2pt,color=qqqqff] (-7,1.5)-- (-6.5,1.5);\draw [line width=1.2pt,color=qqqqff] (-7,2)-- (-6.5,2);\draw [line width=1.2pt,color=qqqqff] (-7,2.5)-- (-6.5,2.5);\draw [line width=1.2pt,color=qqqqff] (-7,3)-- (-6.5,3);\draw [line width=1.2pt,color=qqqqff] (-7,3.5)-- (-6.5,3.5);\draw [line width=1.2pt,color=qqqqff] (-7,4)-- (-6.5,4);\draw [line width=1.2pt,color=qqqqff] (-6,0)-- (-6,0.5);\draw [line width=1.2pt,color=qqqqff] (-6,0.5)-- (-6,1);\draw [line width=1.2pt,color=qqqqff] (-6,1)-- (-5.5,1);\draw [line width=1.2pt,color=qqqqff] (-5.5,1)-- (-5.5,0.5);\draw [line width=1.2pt,color=qqqqff] (-5.5,0.5)-- (-5.5,0);\draw [line width=1.2pt,color=qqqqff] (-6,0.5)-- (-5.5,0.5);\draw [line width=1.2pt,color=qqqqff] (-5.5,0.5)-- (-5,0.5);\draw [line width=1.2pt,color=qqqqff] (-5,0.5)-- (-5,0);\draw [line width=1.2pt,color=qqqqff] (-4.5,0)-- (-4.5,0.5);\draw [line width=1.2pt,color=qqqqff] (-4.5,0.5)-- (-4.5,1);\draw [line width=1.2pt,color=qqqqff] (-4.5,1)-- (-4.5,1.5);\draw [line width=1.2pt,color=qqqqff] (-4.5,1.5)-- (-4.5,2);\draw [line width=1.2pt,color=qqqqff] (-4,0)-- (-4,0.5);\draw [line width=1.2pt,color=qqqqff] (-4,0.5)-- (-4,1);\draw [line width=1.2pt,color=qqqqff] (-4,1)-- (-4,1.5);\draw [line width=1.2pt,color=qqqqff] (-4,1.5)-- (-4,2);
\draw [line width=1.2pt,color=qqqqff] (-4,2)-- (-4.5,2);
\draw [line width=1.2pt,color=qqqqff] (-4.5,1.5)-- (-4,1.5);\draw [line width=1.2pt,color=qqqqff] (-4.5,1)-- (-4,1);\draw [line width=1.2pt,color=qqqqff] (-4.5,0.5)-- (-4,0.5);\draw [line width=1.2pt,color=qqqqff] (-4,1.5)-- (-3.5,1.5);\draw [line width=1.2pt,color=qqqqff] (-3.5,1.5)-- (-3.5,1);\draw [line width=1.2pt,color=qqqqff] (-3.5,1)-- (-3.5,0.5);\draw [line width=1.2pt,color=qqqqff] (-3.5,0.5)-- (-3.5,0);\draw [line width=1.2pt,color=qqqqff] (-4,0.5)-- (-3.5,0.5);\draw [line width=1.2pt,color=qqqqff] (-4,1)-- (-3.5,1);\draw [line width=1.2pt,color=qqqqff] (-3.5,0.5)-- (-3,0.5);\draw [line width=1.2pt,color=qqqqff] (-3,0.5)-- (-3,0);\draw [line width=1.pt,dotted] (-6.5,4.5)-- (-2,0);\draw [line width=1.pt,dotted] (-3,0.5)-- (-2.5,0);\draw [line width=1.pt,dotted] (-5.5,1)-- (-4.5,0);\draw [line width=1.pt,dotted] (2.5,2)-- (0.5,0);
\draw [line width=1.2pt,color=ffqqqq] (4.5,0.5)-- (4.5,1);
\draw [line width=1.2pt,color=ffqqqq] (4.5,1)-- (5,1);
\draw [line width=1.2pt,color=ffqqqq] (5,1)-- (5,0.5);
\draw [line width=1.pt,dotted] (4.5,1)-- (3.5,0);
\end{tikzpicture}}}
\caption{\footnotesize Representation of the Tamari diagram $9021043100$ (left) and of the dual Tamari diagram $0010040002$ (right), both of size $10$.}
\label{dtdtd}
\end{figure}
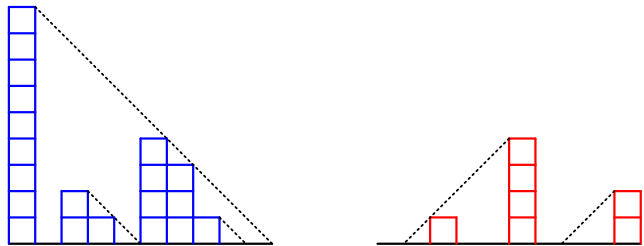

\begin{definition}\label{compatibleDT}
Let $u$ be a Tamari diagram of size $n$ and $v$ be a dual Tamari diagram  of size $n$. The diagrams $u$ and $v$ are \Def{compatible} if for all $1\leq i < j\leq n$ such that $j-i\leq u_i$, we have $v_j < j-i$. 

If $u$ and $v$ are compatible, then the pair $(u,v)$ is a \Def{Tamari interval diagram}. The set of Tamari interval diagrams of size $n$ is denoted by $\dit$.
\end{definition}

For example, the two diagrams of Figure~\ref{dtdtd} are compatible.
Note that Definition~\ref{compatibleDT} implies in particular that either $u_i = 0$ or $v_{i+1} = 0$.

As previously, we draw Tamari interval diagrams through columns, as shown in Figure~\ref{dit10} which displays the Tamari interval diagram $(9021043100, 0010040002)$.
The Tamari diagram $u$ is drawn in blue (under) and its dual $v$ is drawn in 
red (over). 

\definecolor{rvwvcq}{rgb}{0.08235294117647059,0.396078431372549,0.7529411764705882}\definecolor{ffqqqq}{rgb}{1,0,0}\definecolor{qqqqff}{rgb}{0,0,1}

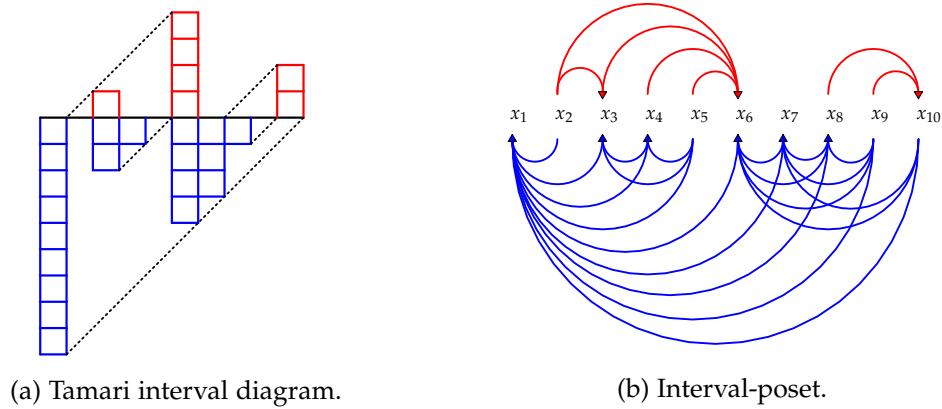
\begin{figure}[h!]
\subfloat[][Tamari interval diagram.]{
\begin{minipage}[c]{.6\textwidth}
\scalebox{0.7}
{\begin{tikzpicture}[line cap=round,line join=round,>=triangle 45,x=1cm,y=1cm]\clip(-4.5,-4.7) rectangle (13.18,2.3);\draw [line width=1.2pt,color=qqqqff] (0,0)-- (0,-0.5);\draw [line width=1.2pt,color=qqqqff] (0,-0.5)-- (0,-1);\draw [line width=1.2pt,color=qqqqff] (0,-1)-- (0,-1.5);\draw [line width=1.2pt,color=qqqqff] (0,-1.5)-- (0,-2);\draw [line width=1.2pt,color=qqqqff] (0,-2)-- (0,-2.5);\draw [line width=1.2pt,color=qqqqff] (0,-2.5)-- (0,-3);\draw [line width=1.2pt,color=qqqqff] (0,-3)-- (0,-3.5);\draw [line width=1.2pt,color=qqqqff] (0,-3.5)-- (0,-4);\draw [line width=1.2pt,color=qqqqff] (0,-4)-- (0,-4.5);\draw [line width=1.2pt,color=qqqqff] (0,-4.5)-- (0.5,-4.5);\draw [line width=1.2pt,color=qqqqff] (0.5,-4.5)-- (0.5,-4);\draw [line width=1.2pt,color=qqqqff] (0.5,-4)-- (0.5,-3.5);\draw [line width=1.2pt,color=qqqqff] (0.5,-3.5)-- (0.5,-3);\draw [line width=1.2pt,color=qqqqff] (0.5,-3)-- (0.5,-2.5);\draw [line width=1.2pt,color=qqqqff] (0.5,-2.5)-- (0.5,-2);\draw [line width=1.2pt,color=qqqqff] (0.5,-2)-- (0.5,-1.5);\draw [line width=1.2pt,color=qqqqff] (0.5,-1.5)-- (0.5,-1);\draw [line width=1.2pt,color=qqqqff] (0.5,-1)-- (0.5,-0.5);\draw [line width=1.2pt,color=qqqqff] (0.5,-0.5)-- (0.5,0);\draw [line width=1.2pt] (0,0)-- (5,0);\draw [line width=1.2pt,color=qqqqff] (0,-0.5)-- (0.5,-0.5);\draw [line width=1.2pt,color=qqqqff] (0,-1)-- (0.5,-1);\draw [line width=1.2pt,color=qqqqff] (0,-1.5)-- (0.5,-1.5);\draw [line width=1.2pt,color=qqqqff] (0,-2)-- (0.5,-2);\draw [line width=1.2pt,color=qqqqff] (0,-2.5)-- (0.5,-2.5);\draw [line width=1.2pt,color=qqqqff] (0,-3)-- (0.5,-3);\draw [line width=1.2pt,color=qqqqff] (0,-3.5)-- (0.5,-3.5);\draw [line width=1.2pt,color=qqqqff] (0,-4)-- (0.5,-4);\draw [line width=1.2pt,color=qqqqff] (1,0)-- (1,-0.5);\draw [line width=1.2pt,color=qqqqff] (1,-0.5)-- (1,-1);\draw [line width=1.2pt,color=qqqqff] (1,-1)-- (1.5,-1);\draw [line width=1.2pt,color=qqqqff] (1.5,-1)-- (1.5,-0.5);\draw [line width=1.2pt,color=qqqqff] (1.5,-0.5)-- (1,-0.5);\draw [line width=1.2pt,color=qqqqff] (1.5,0)-- (1.5,-0.5);\draw [line width=1.2pt,color=qqqqff] (1.5,-0.5)-- (2,-0.5);\draw [line width=1.2pt,color=qqqqff] (2,-0.5)-- (2,0);\draw [line width=1.2pt,color=qqqqff] (2.5,0)-- (2.5,-0.5);\draw [line width=1.2pt,color=qqqqff] (2.5,-0.5)-- (2.5,-1);\draw [line width=1.2pt,color=qqqqff] (2.5,-1)-- (2.5,-1.5);\draw [line width=1.2pt,color=qqqqff] (2.5,-1.5)-- (2.5,-2);\draw [line width=1.2pt,color=qqqqff] (2.5,-2)-- (3,-2);\draw [line width=1.2pt,color=qqqqff] (3,-2)-- (3,-1.5);\draw [line width=1.2pt,color=qqqqff] (3,-1.5)-- (2.5,-1.5);\draw [line width=1.2pt,color=qqqqff] (3,-1.5)-- (3,-1);\draw [line width=1.2pt,color=qqqqff] (3,-1)-- (2.5,-1);\draw [line width=1.2pt,color=qqqqff] (2.5,-0.5)-- (3,-0.5);\draw [line width=1.2pt,color=qqqqff] (3,-0.5)-- (3,-1);\draw [line width=1.2pt,color=qqqqff] (3,-0.5)-- (3,0);\draw [line width=1.2pt,color=qqqqff] (3,-1.5)-- (3.5,-1.5);\draw [line width=1.2pt,color=qqqqff] (3.5,-1.5)-- (3.5,-1);\draw [line width=1.2pt,color=qqqqff] (3.5,-1)-- (3,-1);\draw [line width=1.2pt,color=qqqqff] (3.5,-1)-- (3.5,-0.5);\draw [line width=1.2pt,color=qqqqff] (3.5,-0.5)-- (3.5,0);\draw [line width=1.2pt,color=qqqqff] (3.5,-0.5)-- (4,-0.5);\draw [line width=1.2pt,color=qqqqff] (4,-0.5)-- (4,0);\draw [line width=1.2pt,color=qqqqff] (3,-0.5)-- (3.5,-0.5);\draw [line width=1.2pt,color=ffqqqq] (4.5,0)-- (4.5,0.5);\draw [line width=1.2pt,color=ffqqqq] (4.5,0.5)-- (5,0.5);\draw [line width=1.2pt,color=ffqqqq] (5,0.5)-- (5,0);
\draw [line width=1.2pt,color=ffqqqq] (4.5,0.5)-- (4.5,1);
\draw [line width=1.2pt,color=ffqqqq] (4.5,1)-- (5,1);
\draw [line width=1.2pt,color=ffqqqq] (5,1)-- (5,0.5);
\draw [line width=1.pt,dotted] (4.5,1)-- (3.5,0);
\draw [line width=1.2pt,color=ffqqqq] (3,0)-- (3,0.5);\draw [line width=1.2pt,color=ffqqqq] (3,0.5)-- (3,1);\draw [line width=1.2pt,color=ffqqqq] (3,1)-- (3,1.5);\draw [line width=1.2pt,color=ffqqqq] (3,1.5)-- (3,2);\draw [line width=1.2pt,color=ffqqqq] (3,2)-- (2.5,2);\draw [line width=1.2pt,color=ffqqqq] (2.5,2)-- (2.5,1.5);\draw [line width=1.2pt,color=ffqqqq] (2.5,1.5)-- (2.5,1);\draw [line width=1.2pt,color=ffqqqq] (2.5,1)-- (2.5,0.5);\draw [line width=1.2pt,color=ffqqqq] (2.5,0.5)-- (2.5,0);\draw [line width=1.2pt,color=ffqqqq] (2.5,0.5)-- (3,0.5);\draw [line width=1.2pt,color=ffqqqq] (2.5,1)-- (3,1);\draw [line width=1.2pt,color=ffqqqq] (2.5,1.5)-- (3,1.5);\draw [line width=1.2pt,color=ffqqqq] (1,0)-- (1,0.5);\draw [line width=1.2pt,color=ffqqqq] (1,0.5)-- (1.5,0.5);\draw [line width=1.2pt,color=ffqqqq] (1.5,0.5)-- (1.5,0);
\draw [line width=1.pt,dotted] (2.5,2)-- (0.5,0);
\draw [line width=1.pt,dotted] (4,-0.5)-- (4.5,0);
\draw [line width=1.pt,dotted] (0.5,-4.5)-- (5,0);
\draw [line width=1.pt,dotted] (1.5,-1)-- (2.5,0);
\end{tikzpicture}}
\end{minipage}
\label{dit10}}
%%%%%%%%%%%%%%%%%%%%%%%%%%%%%%%%
\subfloat[][Interval-poset.]{
\begin{minipage}[c]{.25\textwidth}
\hspace*{-1cm}
\scalebox{.6}
{\begin{tikzpicture}[line cap=round,line join=round,>=triangle 45,x=1cm,y=1cm]
\clip(-5.3,-4.97) rectangle (4.8,3.13);
\draw (-5.2,0.8) node[anchor=north west] {$x_1$};
\draw (-4.2,0.8) node[anchor=north west] {$x_2$};
\draw (-3.2,0.8) node[anchor=north west] {$x_3$};
\draw (-2.2,0.8) node[anchor=north west] {$x_4$};
\draw (-1.2,0.8) node[anchor=north west] {$x_5$};
\draw (-0.2,0.8) node[anchor=north west] {$x_6$};
\draw (0.8,0.8) node[anchor=north west] {$x_7$};
\draw (1.8,0.8) node[anchor=north west] {$x_8$};
\draw (2.8,0.8) node[anchor=north west] {$x_9$};
\draw (3.8,0.8) node[anchor=north west] {$x_{10}$};
\draw [shift={(-0.5,-0.03983480176211452)},line width=1.2pt,color=qqqqff]  plot[domain=-3.150444600548006:0.008851946958212795,variable=\t]({1*4.500176308927399*cos(\t r)+0*4.500176308927399*sin(\t r)},{0*4.500176308927399*cos(\t r)+1*4.500176308927399*sin(\t r)});\draw [shift={(-2,0)},line width=1.2pt,color=qqqqff]  plot[domain=-3.141592653589793:0,variable=\t]({1*1*cos(\t r)+0*1*sin(\t r)},{0*1*cos(\t r)+1*1*sin(\t r)});\draw [shift={(-1.5,-0.0009)},line width=1.2pt,color=qqqqff]  plot[domain=-3.143392651645797:0.0017999980560038029,variable=\t]({1*0.5000008099993439*cos(\t r)+0*0.5000008099993439*sin(\t r)},{0*0.5000008099993439*cos(\t r)+1*0.5000008099993439*sin(\t r)});\draw [shift={(-2.5,-0.0001)},line width=1.2pt,color=qqqqff]  plot[domain=-3.141792653587127:0.00019999999733375545,variable=\t]({1*0.5000000099999999*cos(\t r)+0*0.5000000099999999*sin(\t r)},{0*0.5000000099999999*cos(\t r)+1*0.5000000099999999*sin(\t r)});\draw [shift={(2,0)},line width=1.2pt,color=qqqqff]  plot[domain=-3.141592653589793:0,variable=\t]({1*2*cos(\t r)+0*2*sin(\t r)},{0*2*cos(\t r)+1*2*sin(\t r)});\draw [shift={(2.5,-0.04211038961038971)},line width=1.2pt,color=qqqqff]  plot[domain=-3.1696588749668884:0.028066221377095246,variable=\t]({1*1.5005909785524962*cos(\t r)+0*1.5005909785524962*sin(\t r)},{0*1.5005909785524962*cos(\t r)+1*1.5005909785524962*sin(\t r)});\draw [shift={(2.5,-0.020480769230769108)},line width=1.2pt,color=qqqqff]  plot[domain=-3.182531306013036:0.04093865242324278,variable=\t]({1*0.5004192861074446*cos(\t r)+0*0.5004192861074446*sin(\t r)},{0*0.5004192861074446*cos(\t r)+1*0.5004192861074446*sin(\t r)});\draw [shift={(-4.5,-0.0001)},line width=1.2pt,color=qqqqff]  plot[domain=-3.141792653587126:0.00019999999733286727,variable=\t]({1*0.5000000099999999*cos(\t r)+0*0.5000000099999999*sin(\t r)},{0*0.5000000099999999*cos(\t r)+1*0.5000000099999999*sin(\t r)});\draw [shift={(-4,0)},line width=1.2pt,color=qqqqff]  plot[domain=-3.141592653589793:0,variable=\t]({1*1*cos(\t r)+0*1*sin(\t r)},{0*1*cos(\t r)+1*1*sin(\t r)});\draw [shift={(-3.5,-0.000033333333333255645)},line width=1.2pt,color=qqqqff]  plot[domain=-3.1416148758120115:0.00002222222221851245,variable=\t]({1*1.5000000003703704*cos(\t r)+0*1.5000000003703704*sin(\t r)},{0*1.5000000003703704*cos(\t r)+1*1.5000000003703704*sin(\t r)});\draw [shift={(-3,0)},line width=1.2pt,color=qqqqff]  plot[domain=-3.141592653589793:0,variable=\t]({1*2*cos(\t r)+0*2*sin(\t r)},{0*2*cos(\t r)+1*2*sin(\t r)});\draw [shift={(-2.5,-0.0005)},line width=1.2pt,color=qqqqff]  plot[domain=-3.1417926535871263:0.00019999999733337885,variable=\t]({1*2.5000000499999997*cos(\t r)+0*2.5000000499999997*sin(\t r)},{0*2.5000000499999997*cos(\t r)+1*2.5000000499999997*sin(\t r)});\draw [shift={(-2,0)},line width=1.2pt,color=qqqqff]  plot[domain=-3.141592653589793:0,variable=\t]({1*3*cos(\t r)+0*3*sin(\t r)},{0*3*cos(\t r)+1*3*sin(\t r)});\draw [shift={(-1.5,0.04021676300578028)},line width=1.2pt,color=qqqqff]  plot[domain=3.1530826516423014:6.271695309127078,variable=\t]({1*3.500231047806225*cos(\t r)+0*3.500231047806225*sin(\t r)},{0*3.500231047806225*cos(\t r)+1*3.500231047806225*sin(\t r)});\draw [shift={(-1,0)},line width=1.2pt,color=qqqqff]  plot[domain=-3.141592653589793:0,variable=\t]({1*4*cos(\t r)+0*4*sin(\t r)},{0*4*cos(\t r)+1*4*sin(\t r)});\draw [shift={(0.5,-0.01971153846153845)},line width=1.2pt,color=qqqqff]  plot[domain=-3.1809953260308115:0.039402672441018395,variable=\t]({1*0.5003883938987002*cos(\t r)+0*0.5003883938987002*sin(\t r)},{0*0.5003883938987002*cos(\t r)+1*0.5003883938987002*sin(\t r)});\draw [shift={(1,0)},line width=1.2pt,color=qqqqff]  plot[domain=-3.141592653589793:0,variable=\t]({1*1*cos(\t r)+0*1*sin(\t r)},{0*1*cos(\t r)+1*1*sin(\t r)});\draw [shift={(1.5,-0.03951298701298682)},line width=1.2pt,color=qqqqff]  plot[domain=-3.1679285545601936:0.026335900970400596,variable=\t]({1*1.5005203351313465*cos(\t r)+0*1.5005203351313465*sin(\t r)},{0*1.5005203351313465*cos(\t r)+1*1.5005203351313465*sin(\t r)});\draw [shift={(2,0)},line width=1.2pt,color=qqqqff]  plot[domain=-3.141592653589793:0,variable=\t]({1*1*cos(\t r)+0*1*sin(\t r)},{0*1*cos(\t r)+1*1*sin(\t r)});\draw [shift={(1.5,0.04163043478260845)},line width=1.2pt,color=qqqqff]  plot[domain=3.224661921028796:6.200116039740583,variable=\t]({1*0.5017300998546821*cos(\t r)+0*0.5017300998546821*sin(\t r)},{0*0.5017300998546821*cos(\t r)+1*0.5017300998546821*sin(\t r)});\draw [shift={(3,1)},line width=1.2pt,color=ffqqqq]  plot[domain=0:3.141592653589793,variable=\t]({1*1*cos(\t r)+0*1*sin(\t r)},{0*1*cos(\t r)+1*1*sin(\t r)});\draw [shift={(3.5,1.0049)},line width=1.2pt,color=ffqqqq]  plot[domain=-0.009799686287411014:3.1513923398772037,variable=\t]({1*0.5000240094235472*cos(\t r)+0*0.5000240094235472*sin(\t r)},{0*0.5000240094235472*cos(\t r)+1*0.5000240094235472*sin(\t r)});\draw [shift={(-3.5,1.0745689655172412)},line width=1.2pt,color=ffqqqq]  plot[domain=-0.14804674194062262:3.2896393955304157,variable=\t]({1*0.5055299502683412*cos(\t r)+0*0.5055299502683412*sin(\t r)},{0*0.5055299502683412*cos(\t r)+1*0.5055299502683412*sin(\t r)});\draw [shift={(-2,1)},line width=1.2pt,color=ffqqqq]  plot[domain=0:3.141592653589793,variable=\t]({1*2*cos(\t r)+0*2*sin(\t r)},{0*2*cos(\t r)+1*2*sin(\t r)});\draw [shift={(-1.5,1.0003)},line width=1.2pt,color=ffqqqq]  plot[domain=-0.00019999999733322227:3.1417926535871263,variable=\t]({1*1.5000000299999998*cos(\t r)+0*1.5000000299999998*sin(\t r)},{0*1.5000000299999998*cos(\t r)+1*1.5000000299999998*sin(\t r)});\draw [shift={(-1,1)},line width=1.2pt,color=ffqqqq]  plot[domain=0:3.141592653589793,variable=\t]({1*1*cos(\t r)+0*1*sin(\t r)},{0*1*cos(\t r)+1*1*sin(\t r)});\draw [shift={(-0.5,1.0009)},line width=1.2pt,color=ffqqqq]  plot[domain=-0.0017999980560032824:3.1433926516457964,variable=\t]({1*0.5000008099993439*cos(\t r)+0*0.5000008099993439*sin(\t r)},{0*0.5000008099993439*cos(\t r)+1*0.5000008099993439*sin(\t r)});\begin{scriptsize}\draw [fill=qqqqff,shift={(-5,0)}] (0,0) ++(0 pt,3pt) -- ++(2.598076211353316pt,-4.5pt)--++(-5.196152422706632pt,0 pt) -- ++(2.598076211353316pt,4.5pt);\draw [fill=qqqqff,shift={(-3,0)}] (0,0) ++(0 pt,3pt) -- ++(2.598076211353316pt,-4.5pt)--++(-5.196152422706632pt,0 pt) -- ++(2.598076211353316pt,4.5pt);\draw [fill=qqqqff,shift={(-2,0)}] (0,0) ++(0 pt,3pt) -- ++(2.598076211353316pt,-4.5pt)--++(-5.196152422706632pt,0 pt) -- ++(2.598076211353316pt,4.5pt);\draw [fill=qqqqff,shift={(0,0)}] (0,0) ++(0 pt,3pt) -- ++(2.598076211353316pt,-4.5pt)--++(-5.196152422706632pt,0 pt) -- ++(2.598076211353316pt,4.5pt);\draw [fill=qqqqff,shift={(1,0)}] (0,0) ++(0 pt,3pt) -- ++(2.598076211353316pt,-4.5pt)--++(-5.196152422706632pt,0 pt) -- ++(2.598076211353316pt,4.5pt);\draw [fill=qqqqff,shift={(2,0)}] (0,0) ++(0 pt,3pt) -- ++(2.598076211353316pt,-4.5pt)--++(-5.196152422706632pt,0 pt) -- ++(2.598076211353316pt,4.5pt);\draw [fill=ffqqqq,shift={(4,1)},rotate=180] (0,0) ++(0 pt,3pt) -- ++(2.598076211353316pt,-4.5pt)--++(-5.196152422706632pt,0 pt) -- ++(2.598076211353316pt,4.5pt);\draw [fill=ffqqqq,shift={(-3,1)},rotate=180] (0,0) ++(0 pt,3pt) -- ++(2.598076211353316pt,-4.5pt)--++(-5.196152422706632pt,0 pt) -- ++(2.598076211353316pt,4.5pt);\draw [fill=ffqqqq,shift={(0,1)},rotate=180] (0,0) ++(0 pt,3pt) -- ++(2.598076211353316pt,-4.5pt)--++(-5.196152422706632pt,0 pt) -- ++(2.598076211353316pt,4.5pt);\end{scriptsize}\end{tikzpicture}}
\end{minipage}
\label{ip10}}
\caption{\footnotesize Representation of the Tamari interval diagram $(9021043100, 0010040002)$ of size $10$ and its corresponding interval-poset.}
\label{dit}
\end{figure}

Let $\chi$ be the map sending a Tamari interval diagram $(u, v)$ of size $n$ to the binary relation 
\begin{equation}
    \left(\{x_1, \dots, x_n\}, \lhd\right),
\end{equation}
where for all $i \in [n]$ 
and $0\leq l\leq u_{i}$,
$x_{i+l}\lhd x_{i}$,
and for all $i\in [n]$ and $0\leq k\leq v_{i}$,
$x_{i-k}\lhd x_{i}$.

We recall that an \Def{interval-poset} $P$ of size $n$ is a partial order $\lhd$ on the set $\{x_1, \dots,x_n\}$ such that, for $i < k$, if $x_k\lhd x_i$ then for all $x_j$ such that $i<j<k$, one has $x_j\lhd x_i$, and if $x_i\lhd x_k$ then for all $x_j$ such that $i<j<k$, one has $x_j\lhd x_k$~\cite{CP15} (see Figure~\ref{ip10} for instance).
We denote $\ip$ the set of interval-posets of size $n$.

\begin{Theorem}
The application $\chi$ is a bijection from $\dit$ to $\ip$.
\end{Theorem}

\begin{proof}
We set $P = \chi(u,v)$. First, we check that $P$ satisfies all axioms of the definition of interval-posets.
Then we show that $\chi$ is surjective: for every interval-poset $P$, the pair of words $(u_1 u_2 \dots u_n , v_1 v_2 \dots v_n) \in \mathbb{N}^n\times\mathbb{N}^n$, where for all $i, j \in[n]$,
\begin{align}\label{réciproque}
%\begin{split}
    u_i = \# \{x_j\in P : x_j\lhd x_i \mbox{ and } i<j\}, \\
    v_j = \# \{x_i\in P : x_i\lhd x_j \mbox{ and } i<j\},\label{réciproque2}
%\end{split}
\end{align}
has image $P$ and is a Tamari interval diagram. 
The proof that $\chi$ is injective is direct.
\end{proof}

%%%%%%%%%%%%%%%%%%%%%%%%%%%%%%%%%%%%%%%%%%%%%%%%
%%%%%%%%%%%%%%%%%%%%%%%%%%%%%%%%%%%%%%%%%%%%%%%%
\section{Cubic coordinates}\label{Part3}

We now build the set of cubic coordinates and  provide a bijection between this set and the set of Tamari interval diagrams. We conclude this section by reviewing some properties of cubic coordinates.
\smallbreak

Let $(u,v)\in\dit$. We build a $(n-1)$-tuple $(u_1-v_2, u_2-v_3,\dots, u_{n-1}-v_n)$ from letters of $(u,v)$.
This $(n-1)$-tuple can be defined by using the definition of Tamari interval diagrams.

\begin{definition}\label{coordCubic}
Let $c$ be a $(n-1)$-tuple with entries in $\mathbb{Z}$. We say that $c$ is a \Def{cubic coordinate} if the pair $(u, v)$, where
$u$ is the word defined by $u_n = 0$ and for all $i\in[n-1]$ by
\begin{equation*}
    u_i = \max (c_i,~0),
\end{equation*}
and $v$ is the word defined by $v_1 = 0$ and
for all $2\leq i\leq n$ by
\begin{equation*}
    v_i = |\min (c_{i-1},~0)|,
\end{equation*}
is a Tamari interval diagram. The size of a cubic coordinate is its number of entries plus one.
The set of cubic coordinates of size $n$ is denoted by $\cc$.
\end{definition}

For example, in Figure~\ref{dit}, the Tamari interval diagram has for cubic coordinate $(9,-1,2,1,-4,4,3,1,-2)$.

Let $\phi$ be the application which maps a cubic coordinate $c$ to a Tamari interval diagram as stated in Definition~\ref{coordCubic}.

\begin{Theorem}\label{bijDIT-CC}
The application $\phi : \cc \rightarrow \dit$ is bijective.
\end{Theorem}

\begin{proof}
Let $c, c'\in\cc$ such that $c\neq c'$. Then, there is an entry $c_i\neq c'_i$, with $i\in[n-1]$. By the definition of the application $\phi$, there is a letter $u_i\neq u'_i$, or a letter $v_{i+1}\neq v'_{i+1}$, meaning that $(u,v)\neq (u',v')$. This shows the injectivity of $\phi$.

Let $(u,v)\in\dit$ and let $c = (u_1-v_2, u_2-v_3,\dots, u_{n-1}-v_n)$ be the $(n-1)$-tuple whose entries are given by the difference $u_i-v_{i+1}$, for all $i\in[n-1]$. If $u_i\neq 0$ then by Definition~\ref{compatibleDT}, $v_{i+1} = 0$, so we have $\phi(c) = (u,v)$. As $(u,v)$ is a Tamari interval diagram by hypothesis and by Definition~\ref{coordCubic}, the application $\phi$ is surjective.
\end{proof}

\begin{Lemma}\label{lemmeutilepartout}
Let $c \in\cc$ with $c_i \ne 0$ for some $i\in[n-1]$. We denote by $c'$ the $(n-1)$-tuple such that $c'_i = 0$ and all entries having indices different from $i$ are the ones of $c$. Then $c'$ is a cubic coordinate.
\end{Lemma}

\begin{proof}
We set $(u,v)=\phi(c')$ and $c'_i=0$. Then we can check all axioms of Definitions~\ref{diagTam},~\ref{diagTamDual}, and~\ref{compatibleDT} for $(u,v)$ with the pair of letters $(u_i,v_{i+1})=(0,0)$.
\end{proof}

\begin{definition}\label{ditSync}
A cubic coordinate $c$ of size $n$ is \Def{synchronized} if for all $i\in [n-1]$, $c_i\ne 0$. The set of synchronized cubic coordinates of size $n$ is denoted by $\ccs$.
\end{definition}

This definition can be translated in terms of Tamari interval diagrams: a Tamari interval diagram $(u,v)$ of size $n$ is \Def{synchronized} if for all $i\in [n-1]$, $u_i \ne 0$ or $v_{i+1} \ne 0$.

We recall that a Tamari interval $[S,T]$ is synchronized if and only if the binary trees $S$ and $T$ have the same canopy~\cite{PRV17}. Let $\itam$ be the set of Tamari intervals of size $n$ and $\rho$ be the bijection from $\ip$ to $\itam$~\cite{CP15}.

\begin{Proposition}\label{ditsyn=tisyn}
Let $(u,v)\in\dit$. The Tamari interval diagram $(u,v)$ is synchronized if and only if $\rho(\chi(u,v)) = [S,T]$ is a synchronized Tamari interval. 
\end{Proposition}

\begin{proof}
Arguing by contradiction, we translate the fact that $u_i = 0$ and $v_{i+1} = 0$ to the trees $S$ and $T$, and deduce that their canopies are distinct. Reciprocally, if the canopies of $S$ and of $T$ are distinct, we can find an index $i\in[n]$ such that $u_i = 0$ and $v_{i+1} = 0$.
\end{proof}

\begin{definition}\label{ditnouveau}
A Tamari interval diagram $(u,v)$ of size $n$ is \Def{new} if the following conditions are satisfied:
\begin{enumerate}[label=(\roman*)]
 \item \label{ditnv1} $0\leq u_i \leq n-i-1$ for all $i\in[n-1]$;
 \item \label{ditnv2} $0\leq v_j\leq j-2$ for all $j\in \{2,\dots,n\}$;
 \item \label{ditnv3} $u_{k}<l-k-1$ or $v_{l}<l-k-1$ for all $k,l \in [n]$ such that $k+1 < l$.
 \end{enumerate}
\end{definition} 

The definition of new interval-posets is given in~\cite{Rog18}. The three conditions of this original definition imply the three conditions of Definition~\ref{ditnouveau} and reciprocally. Then, we have the following proposition.

\begin{Proposition}\label{ditnv=intnv}
Let $(u,v)\in\dit$. The Tamari interval diagram $(u,v)$ is new if and only if $\chi(u,v) = P$ is a new interval-poset.
\end{Proposition}

In~\cite{Rog18}, Rognerud shows that an interval-poset $P$ is new if and only if $\rho(P)$ is a new Tamari interval (see~\cite{Cha17} for more about this notion). Then, one has the following result.

\begin{Corollary}
Let $(u,v)\in\dit$. The Tamari interval diagram $(u,v)$ is new if and only if $\rho(\chi(u,v))$ is a new Tamari interval. 
\end{Corollary}

\begin{Proposition}\label{synnn}
Let $(u,v)\in\dit$. If $(u,v)$ is synchronized then $(u,v)$ is not new.
\end{Proposition}

\begin{proof}
Let us suppose that $(u,v)$ is synchronized and new. Then, by using~\ref{ditnv3} from Definition~\ref{ditnouveau}, we come to a contradiction.
\end{proof}

%%%%%%%%%%%%%%%%%%%%%%%%%%%%%%%%%%%%%%%%%%%%%%%%%%%%
%%%%%%%%%%%%%%%%%%%%%%%%%%%%%%%%%%%%%%%%%%%%%%%%%%%%%
\section{Isomorphism of posets}\label{Part4}

In this section, we define the poset of cubic coordinates and we show that there is an isomorphism between this poset and the poset of Tamari intervals. To this aim, both bijections seen in Section~\ref{Part2} and in Section~\ref{Part3} are used.
\smallbreak

Let $c, c'\in\cc$. We set $c\leq_{\mathrm{cc}} c'$ if and only if $c_i\leq c'_i$ for all $i\in[n-1]$.
The set of cubic coordinates of size $n$ endowed with the binary relation $\leq_{\mathrm{cc}}$ is a poset, the \Def{cubic coordinate poset}.

Let us denote by $\lessdot$ the covering relation of $(\cc,\leq_{\mathrm{cc}})$. We admit here the following result.

\begin{Lemma}
Let $c,c'\in\cc$ such that $c\leq_{\mathrm{cc}}c'$. Then, 
$c\lessdot c'$ if and only if there is exactly one $i\in[n-1]$ such that $c_i < c'_i$, and if there is a $c''\in\cc$ such that $c\leq_{\mathrm{cc}}c''\leq_{\mathrm{cc}}c'$, then either $c = c''$ or $c' =c''$.
\end{Lemma}

We denote by $\leq_{\mathrm{t}}$ the Tamari order on the set of binary trees~\cite{Tam62}: for $S$ and $T$, two binary trees of size $n$, $S \leq_{\mathrm{t}} T$ if and only if $T$ can be obtained by performing an arbitrary number of rotations from $S$. If $T$ is obtained by only one rotation in $S$, then $T$ covers $S$. Let $[S,T], [S',T']\in \itam$. We set $[S,T] \leq_{\mathrm{ti}}[S',T']$ if and only if $S \leq_{\mathrm{t}} S'$ and $ T\leq_{\mathrm{t}}T'$. Then, $S'$ covers $S$ and $T=T'$ or $T'$ covers $T$ and $S=S'$ if and only if $[S',T']$ covers $[S,T]$.

Let $\psi = \phi^{-1} \circ \chi^{-1} \circ \rho^{-1}$ be the application from the set of Tamari intervals to the set of cubic coordinates.

\begin{Theorem}\label{morphPoset}
The application $\psi$ is an isomorphism of posets.
\end{Theorem} 

\begin{proof}
Let $[S,T], [S',T']\in \itam$, and $\psi([S,T]) = c$, $\psi([S',T']) = c'$.
Then, let $\phi(c) = (u,v)$, $\phi(c') = (u',v')$ and $\chi(u,v) = P$, $\chi(u',v') = P'$. 
Let us show that $[S',T']$ covers $[S,T]$ in $(\itam , \leq_{\mathrm{ti}})$ if and only if $c'$ covers $c$ in $(\cc, \leq_{\mathrm{cc}})$.

Let $(\star)$ (resp. $(\diamond)$) be the following condition: $P'$ is obtained from $P$ by only adding (resp. removing) some decreasing (resp. increasing) relations ending at a vertex $x_i$, such that if we remove (resp. add) one of these decreasing (resp. increasing) relations, then either we obtain $P$ or the resulting object is not an interval-poset.

We admit here that $P$ and $P'$ satisfy $(\star)$ (resp. $(\diamond)$) for the vertex $x_i$ if and only if $S'$ (resp. $T'$) is obtained by a unique rotation of the node of index $i$ in $S$ (resp. $T$) and $T' = T$ (resp. $S' = S$). In other words, $P$ and $P'$ satisfy either $(\star)$ or $(\diamond)$ if and only if $[S',T']$ covers $[S,T]$. It only remains to show that $c'$ covers $c$ with $c_i < c'_i$ if and only if $P$ and $P'$ satisfy either $(\star)$ or $(\diamond)$ for the vertex $x_i$.
 
We assume that $c\lessdot c'$ with $c_i < c'_i$. Then, there are two cases:

\begin{enumerate}[label=({\arabic*})]
\item If $c'_i$ is positive, then $c_{i}$ is not negative due to the Lemma~\ref{lemmeutilepartout}.
Hence $c'_{i}= u'_{i}$ and $c_{i}= u_{i}$. The image by $\phi$ of $c$ and of $c'$ differs only for the letter $u_{i}$. 
Besides, the fact that $c\lessdot c'$ implies in particular that if there is a word $u''$ of size $n$ such that $u''_i = u'_{i}-1$ and $u''_j = u'_{j}$ for any $j\ne i$, then either $(u'',v') = (u,v)$ or $(u'',v')$ is not a Tamari interval diagram. Indeed, let us suppose that there is a Tamari interval diagram $(u'',v')$ such as previously described, different from $(u,v)$. Then, let $c''$ be the cubic coordinate associated with $(u'',v')$ by $\phi^{-1}$. Since $u''_{i}=u'_{i}-1$ and $u''_j = u'_{j}$ for any $j\ne i$, one has $c'' \leq_{\mathrm{cc}} c'$. Moreover, since $u''\neq u$, one has $c \leq_{\mathrm{cc}} c''$. Knowing these two facts, we have a contradiction with our hypothesis $c\lessdot c'$.
\newpage
\noindent
The difference of a unique letter $u_{i}$ between $(u,v)$ and $(u',v')$ is directly translated by $\chi$: The interval-poset $P'$ has more decreasing relations ending at $x_i$ than the vertex $x_i$ has in $P$. Moreover, the fact that there is no other Tamari interval diagram between $(u,v)$ and $(u',v')$ implies that the number of decreasing relations added in $P'$ compared to $P$ is minimal. This means that if one decreasing relation ending at $x_i$ is removed from $P'$, then either we obtain $P$ or the resulting object is not an interval-poset. Hence $P$ and $P'$ satisfy $(\star)$.

\item Symmetrically, if $c'_i$ is nonpositive, then by using the same arguments, we obtain that the interval-poset $P'$ has less increasing relations ending at $x_{i+1}$ than the vertex $x_{i+1}$ from $P$, in a minimal way. This leads us to the fact that $P$ and $P'$ satisfy $(\diamond)$. 
\end{enumerate}

Reciprocally, suppose that $P$ and $P'$ satisfy either $(\star)$ or $(\diamond)$ for a vertex $x_i$.

\begin{enumerate}[label=({\arabic*})]
\item\label{Casdécroiss} Suppose that $P$ and $P'$ satisfy $(\star)$. Since only decreasing relations ending at $x_i$ are added in $P'$, then only the letter $u'_i$ from $u'$ is increased compared to $u$, and $v' = v$. Besides, since the number of decreasing relations added in $P$ is minimal, there is no Tamari interval diagram between $(u,v)$ and $(u',v')$, and so there is no cubic coordinate between $c$ and $c'$. Hence $c \lessdot c'$.

\item\label{Cascroiss} Suppose that $P$ and $P'$ satisfy $(\diamond)$. Since only increasing relations ending at $x_i$ are removed in $P'$, then only the letter $v'_i$ from $v'$ is decreased compared to $v$, and $u' = u$. Then, just like in case \ref{Casdécroiss}, $c \lessdot c'$.
\end{enumerate}

One can now conclude that $\psi$ is an isomorphism of posets.
\end{proof}

To sum up the applications seen in Sections~\ref{Part2}, \ref{Part3} and~\ref{Part4}, let us recall that $\psi = \phi^{-1} \circ \chi^{-1} \circ \rho^{-1}$. Then, we have the following diagram of poset isomorphisms:
$$\begin{tikzcd}
\dit \arrow{r}{\chi} & \ip \arrow{d}{\rho} \\
\cc \arrow[swap]{u}{\phi} & \itam \arrow[swap]{l}{\psi}
\end{tikzcd}$$

One consequence of the isomorphism $\psi$ is that the order dimension~\cite{Tro02} of the poset of Tamari intervals is at most equal to $n-1$.
%%%%%%%%%%%%%%%%%%%%%%%%%%%%%%%%%%%%%%%%%%%%%%%%%%%%%%%%%%%%%%%%%%%%%%%%%%%%%%%%%%%%%%%%%%%%%%%%%%%%%
\section{Cubic realization and cells}\label{Part5}
Now, we regard the poset of cubic coordinates defined in
Section~\ref{Part4} as a natural geometric object. 
We study this geometric realization by giving a theoretical definition of the cells it contains.
\smallbreak

All cubic coordinates of size $n$ can be placed in the space $\mathbb{R}^{n-1}$, as space coordinates.
For all cubic coordinates $c$ and $c'$ such that $c\leq_{\mathrm{cc}}c'$, we connect $c$ to $c'$ with an arrow if and only if there is no other cubic coordinate $c''$ such that $c\leq_{\mathrm{cc}}c''\leq_{\mathrm{cc}}c'$, meaning that just one entry increases between $c$ and $c''$. This oriented graph is the \Def{cubic realization} of the cubic coordinate lattice.

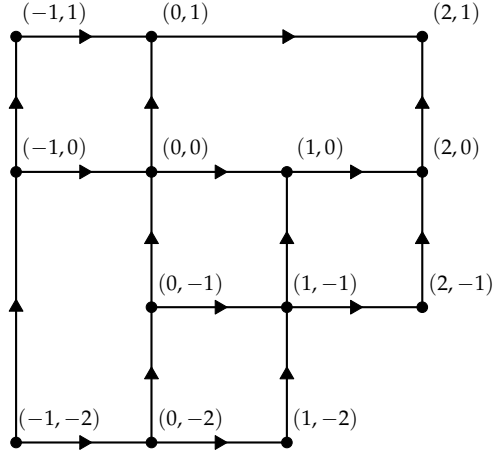
\begin{figure}[h!]
\centerline{\begin{tikzpicture}[line cap=round,line join=round,>=triangle 45,x=1.8cm,y=1.8cm]
\clip(-1.6,-2.1) rectangle (2.7,1.3);
\draw [line width=0.8pt] (0,0)-- (1,0);
\draw [line width=0.8pt] (1,0)-- (2,0);
\draw [line width=0.8pt] (2,0)-- (2,-1);
\draw [line width=0.8pt] (2,-1)-- (1,-1);
\draw [line width=0.8pt] (1,-1)-- (1,0);
\draw [line width=0.8pt] (1,-1)-- (0,-1);
\draw [line width=0.8pt] (0,-1)-- (0,-2);
\draw [line width=0.8pt] (0,-2)-- (1,-2);
\draw [line width=0.8pt] (1,-2)-- (1,-1);
\draw [line width=0.8pt] (0,-2)-- (-1,-2);
\draw [line width=0.8pt] (-1,-2)-- (-1,0);
\draw [line width=0.8pt] (-1,0)-- (0,0);
\draw [line width=0.8pt] (0,0)-- (0,1);
\draw [line width=0.8pt] (0,1)-- (-1,1);
\draw [line width=0.8pt] (-1,1)-- (-1,0);
\draw [line width=0.8pt] (0,1)-- (2,1);
\draw [line width=0.8pt] (2,1)-- (2,0);
\draw [line width=0.8pt] (0,0)-- (0,-1);
\begin{scriptsize}
\draw [fill=black] (0,0) circle (2pt);
\draw[color=black] (0.25111111111110807,0.17333333333334033) node {$(0, 0)$};
\draw [fill=black] (1,0) circle (2pt);
\draw[color=black] (1.2466666666666646,0.17333333333334033) node {$(1, 0)$};
\draw [fill=black] (2,0) circle (2pt);
\draw[color=black] (2.25111111111111,0.17333333333334033) node {$(2, 0)$};
\draw [fill=black] (2,-1) circle (2pt);
\draw[color=black] (2.2866666666666657,-0.8311111111111096) node {$(2, -1)$};
\draw [fill=black] (1,-1) circle (2pt);
\draw[color=black] (1.2822222222222202,-0.8311111111111096) node {$(1, -1)$};
\draw [fill=black] (0,-1) circle (2pt);
\draw[color=black] (0.2866666666666637,-0.8311111111111096) node {$(0, -1)$};
\draw [fill=black] (0,-2) circle (2pt);
\draw[color=black] (0.2866666666666637,-1.8266666666666704) node {$(0, -2)$};
\draw [fill=black] (1,-2) circle (2pt);
\draw[color=black] (1.2822222222222202,-1.8266666666666704) node {$(1, -2)$};
\draw [fill=black] (-1,-2) circle (2pt);
\draw[color=black] (-0.6822222222222262,-1.8266666666666704) node {$(-1, -2)$};
\draw [fill=black] (-1,0) circle (2pt);
\draw[color=black] (-0.7177777777777817,0.17333333333334033) node {$(-1, 0)$};
\draw [fill=black] (0,1) circle (2pt);
\draw[color=black] (0.25111111111110807,1.1688888888889015) node {$(0, 1)$};
\draw [fill=black] (-1,1) circle (2pt);
\draw[color=black] (-0.7177777777777817,1.1688888888889015) node {$(-1, 1)$};
\draw [fill=black] (2,1) circle (2pt);
\draw[color=black] (2.25111111111111,1.1688888888889015) node {$(2, 1)$};
\draw [fill=black,shift={(-1,0.5)}] (0,0) ++(0 pt,3pt) -- ++(2.598076211353316pt,-4.5pt)--++(-5.196152422706632pt,0 pt) -- ++(2.598076211353316pt,4.5pt);
\draw [fill=black,shift={(-0.5,1)},rotate=270] (0,0) ++(0 pt,3pt) -- ++(2.598076211353316pt,-4.5pt)--++(-5.196152422706632pt,0 pt) -- ++(2.598076211353316pt,4.5pt);
\draw [fill=black,shift={(1,1)},rotate=270] (0,0) ++(0 pt,3pt) -- ++(2.598076211353316pt,-4.5pt)--++(-5.196152422706632pt,0 pt) -- ++(2.598076211353316pt,4.5pt);
\draw [fill=black,shift={(2,0.5)}] (0,0) ++(0 pt,3pt) -- ++(2.598076211353316pt,-4.5pt)--++(-5.196152422706632pt,0 pt) -- ++(2.598076211353316pt,4.5pt);
\draw [fill=black,shift={(1.5,0)},rotate=270] (0,0) ++(0 pt,3pt) -- ++(2.598076211353316pt,-4.5pt)--++(-5.196152422706632pt,0 pt) -- ++(2.598076211353316pt,4.5pt);
\draw [fill=black,shift={(0.5,0)},rotate=270] (0,0) ++(0 pt,3pt) -- ++(2.598076211353316pt,-4.5pt)--++(-5.196152422706632pt,0 pt) -- ++(2.598076211353316pt,4.5pt);
\draw [fill=black,shift={(-0.5,0)},rotate=270] (0,0) ++(0 pt,3pt) -- ++(2.598076211353316pt,-4.5pt)--++(-5.196152422706632pt,0 pt) -- ++(2.598076211353316pt,4.5pt);
\draw [fill=black,shift={(-1,-1)}] (0,0) ++(0 pt,3pt) -- ++(2.598076211353316pt,-4.5pt)--++(-5.196152422706632pt,0 pt) -- ++(2.598076211353316pt,4.5pt);
\draw [fill=black,shift={(0,-1.5)}] (0,0) ++(0 pt,3pt) -- ++(2.598076211353316pt,-4.5pt)--++(-5.196152422706632pt,0 pt) -- ++(2.598076211353316pt,4.5pt);
\draw [fill=black,shift={(1,-1.5)}] (0,0) ++(0 pt,3pt) -- ++(2.598076211353316pt,-4.5pt)--++(-5.196152422706632pt,0 pt) -- ++(2.598076211353316pt,4.5pt);
\draw [fill=black,shift={(0.5,-1)},rotate=270] (0,0) ++(0 pt,3pt) -- ++(2.598076211353316pt,-4.5pt)--++(-5.196152422706632pt,0 pt) -- ++(2.598076211353316pt,4.5pt);
\draw [fill=black,shift={(1,-0.5)}] (0,0) ++(0 pt,3pt) -- ++(2.598076211353316pt,-4.5pt)--++(-5.196152422706632pt,0 pt) -- ++(2.598076211353316pt,4.5pt);
\draw [fill=black,shift={(1.5,-1)},rotate=270] (0,0) ++(0 pt,3pt) -- ++(2.598076211353316pt,-4.5pt)--++(-5.196152422706632pt,0 pt) -- ++(2.598076211353316pt,4.5pt);
\draw [fill=black,shift={(2,-0.5)}] (0,0) ++(0 pt,3pt) -- ++(2.598076211353316pt,-4.5pt)--++(-5.196152422706632pt,0 pt) -- ++(2.598076211353316pt,4.5pt);
\draw [fill=black,shift={(0.5,-2)},rotate=270] (0,0) ++(0 pt,3pt) -- ++(2.598076211353316pt,-4.5pt)--++(-5.196152422706632pt,0 pt) -- ++(2.598076211353316pt,4.5pt);
\draw [fill=black,shift={(-0.5,-2)},rotate=270] (0,0) ++(0 pt,3pt) -- ++(2.598076211353316pt,-4.5pt)--++(-5.196152422706632pt,0 pt) -- ++(2.598076211353316pt,4.5pt);
\draw [fill=black,shift={(0,0.5)}] (0,0) ++(0 pt,3pt) -- ++(2.598076211353316pt,-4.5pt)--++(-5.196152422706632pt,0 pt) -- ++(2.598076211353316pt,4.5pt);
\draw [fill=black,shift={(0,-0.5)}] (0,0) ++(0 pt,3pt) -- ++(2.598076211353316pt,-4.5pt)--++(-5.196152422706632pt,0 pt) -- ++(2.598076211353316pt,4.5pt);
\end{scriptsize}
\end{tikzpicture}}
\caption{\footnotesize Cubic realization of $\mathcal{CC}_3$.}
\label{RealCube}
\end{figure}

Figure~\ref{RealCube} is the cubic realization of $\mathcal{CC}_3$, where the elements of $\mathcal{CC}_3$ are vertices and the cover relations are arrows orientated to the covering cubic coordinates.

\begin{definition}\label{augMin}
Let $c\in\cc$. Suppose that there is $c'\in \cc$ such that $c'_i > c_i$ and $c'_j = c_j$ for all $j\ne i$, with $i,j\in [n-1]$.
We define then the application of \Def{minimal increase} $\uparrow_{i}$ as follows
\begin{equation}
\uparrow_{i}(c) = (c_{1},\dots,c_{i-1},\widehat{c}_{i},c_{i+1},\dots,c_{n-1}),
\end{equation}
such that $c ~\lessdot\uparrow_{i}(c)$ and $c_i < \widehat{c}_{i} \leq c'_i$.
\end{definition}

\begin{definition}
Let $c\in\cc$. We say that $c$ is \Def{minimal-cellular} if for all $i\in[n-1]$, $\uparrow_i(c)$ is well-defined.
\end{definition}

We notice that the cubic coordinates which are minimal-cellular are the elements that are covered by exactly $n-1$ elements in $(\cc, \leq_{\mathrm{cc}})$.

\begin{Lemma}\label{existeMax}
Let $c$ be minimal-cellular cubic coordinate of size $n$ and $i\in[n-1]$. If
\begin{equation} \label{equ:existeMax}
	c' = 
	\uparrow_{i+1}(\uparrow_{i+2}(\dots(\uparrow_{n-1}(c))\dots)),
\end{equation}
 is well-defined, then $\uparrow_{i}(c')$ is well-defined.
\end{Lemma}

\begin{definition}\label{defMax}
Let $c$ be a minimal-cellular cubic coordinate of size $n$ and let $c'$ be a cubic coordinate of size $n$. We set that $c'$ is the \Def{maximal-cellular correspondent} of $c$ if 
\begin{equation}
c' = \uparrow_{1}(\uparrow_{2}(\dots(\uparrow_{n-1}(c))\dots)).
\end{equation}
\end{definition}

For instance, $c = (0,-1,1,-1,-5,0,1,-1,-3)$ is minimal-cellular, and its maximal-cellular correspondent is $c' = (1,0,2,0,-4,3,2,0,-2)$.
Such an element always exists, by Lemma~\ref{existeMax}. Note that by performing minimal increases in another order does not always lead to the maximal-cellular correspondent (see Figure~\ref{RealCube} for example).

\begin{definition}
Let $c^{m}$ be minimal-cellular cubic coordinate of size $n$ and let $c^{M}$ be its maximal-cellular correspondent. The pair $(c^{m},c^{M})$ is a \Def{cell}, denoted by $\langle c^{m},c^{M} \rangle$. The size of the cell is the size of $c^{m}$.
\end{definition}

\begin{Lemma}\label{lemmequifaitgagnerdutps}
Let $\langle c^{m},c^{M} \rangle$ be a cell of size $n$ and $i\in [n-1]$, 
\begin{enumerate}[label=(\roman*)]
\item if $c^{m}_{i}<0$ then $c^{M}_{i}\leq 0$;
\item if $c^{m}_{i}\geq 0$ then $c^{M}_{i}>0$.
\end{enumerate}
\end{Lemma}

\begin{Theorem}\label{nombredesommetdanscellule}
Let $\langle c^{m},c^{M} \rangle$ be a cell of size $n$. Let $c$ be any $(n-1)$-tuple whose entries $c_i$ are equal to $c^{m}_i$, or to $c^{M}_i$, for all $i\in [n-1]$. Then $c$ is a cubic coordinate.
\end{Theorem}

One of the consequences of Theorem~\ref{nombredesommetdanscellule} is that for every cell, 
we have at least $2^{n-1}$ cubic coordinates between $c^{m}$ and $c^{M}$. There can be strictly more.
Now, since we have a definition of cells, we show that from a cell, we can build a synchronized cubic coordinate. 

Consider $\langle c^{m},c^{M} \rangle$ a cell of size $n$. Let $\gamma$ be the map defined for all $i\in [n-1]$ by
\begin{equation}
\gamma (c^{m}_{i}, c^{M}_{i}) = 
\begin{cases}
    c^{m}_{i} & \mbox{ if } c^{m}_{i}<0,\\
    c^{M}_{i} & \mbox{ if } c^{m}_{i}\geq 0.
\end{cases}
\end{equation}

Let $\Gamma$ be the map from the set of cells of size $n$ to the set of $(n-1)$-tuples defined by
\begin{align}
\Gamma (\langle c^{m},c^{M} \rangle) = (\gamma(c^{m}_{1}, c^{M}_{1}), \gamma(c^{m}_{2}, c^{M}_{2}),\dots,\gamma(c^{m}_{n-1}, c^{M}_{n-1})).
\end{align}

For example, the cell $\langle (0,-1,1,-1,-5,0,1,-1,-3),(1,0,2,0,-4,3,2,0,-2) \rangle$ is sent to $(1,-1,2,-1,-5,3,2,-1,-3)$.

\begin{Theorem}\label{bigGam}
The application $\Gamma$ is a bijection from the set of cells of size $n$ to $\ccs$.
\end{Theorem}

\begin{proof}
The entries of $\Gamma (\langle c^{m},c^{M} \rangle)$ belong either to $c^{m}$ or to $c^{M}$. In both cases, the entries are not zero. Hence, by Theorem~\ref{nombredesommetdanscellule}, $\Gamma (\langle c^{m},c^{M} \rangle)$ is a cubic coordinate of size~$n$. Additionally, by Definition~\ref{ditSync} and Lemma~\ref{lemmequifaitgagnerdutps}, this cubic coordinate is synchronized.

Let us first show that $\Gamma$ is injective.
Let $\langle c^{m},c^{M} \rangle$ and $\langle e^{m},e^{M} \rangle$ be two cells of size $n$, such that $\Gamma (\langle c^{m},c^{M} \rangle) = \Gamma (\langle e^{m},e^{M} \rangle)$.
%Let $c^{m} = (c^{m}_1, c^{m}_2, \dots, c^{m}_{n-1})$, $c^{M} = (c^{M}_1, c^{M}_2, \dots, c^{M}_{n-1})$, and $e^{m} = (e^{m}_1, e^{m}_2, \dots, e^{m}_{n-1})$, $e^{M} = (e^{M}_1, e^{M}_2, \dots, e^{M}_{n-1})$. 
Then let $(u^j_i,v^j_{i+1})$ and $(x^j_i,y^j_{i+1})$ be the two pairs of letters corresponding respectively to $c^j_i$ and to $e^j_i$ by $\phi$ with $j\in \{m,M\}$ and $i\in[n-1]$.

By hypothesis, $\gamma (c^{m}_i, c^{M}_i) = \gamma (e^{m}_i, e^{M}_i)$ for all $i\in[n-1]$.
We need to prove that $c^j_i = e^j_i$ for all $i\in[n-1]$ with $j\in \{m,M\}$. Thus, two cases are possible:

\begin{enumerate}[label={(\arabic*)}]
\item \label{injGam1} either $\gamma (c^{m}_i, c^{M}_i) = u^{M}_{i}$. In this case, $\gamma (e^{m}_i, e^{M}_i) = x^{M}_{i}$ and $u^{M}_{i} = x^{M}_{i}$;

\item \label{injGam2} or $\gamma (c^{m}_i, c^{M}_i) = -v^{m}_{i+1}$. Then, $\gamma (e^{m}_i, e^{M}_i) = -y^{m}_{i+1}$ and $v^{m}_{i+1} = y^{m}_{i+1}$.
\end{enumerate}

\begin{itemize}
\item Let us assume that~\ref{injGam1} holds. Since $u^{M}_i \ne 0$, by Definition~\ref{compatibleDT}, $v^{M}_{i+1} = 0$. Likewise, $x^{M}_i \ne 0$ implies that $y^{M}_{i+1} = 0$. Hence $c^{M}_i = e^{M}_i$.

\noindent
By Lemma~\ref{lemmequifaitgagnerdutps}, if $u^{M}_i \ne 0$ (resp. $x^{M}_i \ne 0$) then  $0\leq u^{m}_i < u^{M}_i$ and $v^{m}_{i+1} = 0$ (resp. $0\leq x^{m}_i < x^{M}_i$ and $y^{m}_{i+1} = 0$). It remains only to show that $u^{m}_i = x^{m}_i$. 
Suppose that $u^{m}_i < x^{m}_i$, and 
let $c =\uparrow_{i+1}(\dots (\uparrow_{n-1}(c^{m}))\dots)$ and $e = \uparrow_{i+1}(\dots (\uparrow_{n-1}(e^{m}))\dots)$. By Lemma~\ref{existeMax}, $c$ and $e$ are cubic coordinates. Therefore, by Definition~\ref{defMax}, $c_{k}= c^{M}_k$ and $e_{k} = e^{M}_k$ for all $i+1\leq k\leq n-1$. By hypothesis, if $c_k$ is positive, then $c_k = u^{M}_k$ and since in this case $u^{M}_k = x^{M}_k$, one has $e_k = x^{M}_k$.
Let $c'$ be a $(n-1)$-tuple such that $c'_i = x^{m}_i$ and $c'_j = c_j$ for any $j\ne i$. By hypothesis, we also know that $\widehat{c}_i = u^{M}_i$, $\widehat{e}_i = x^{M}_i$, and $u^{M}_i = x^{M}_i$. Since $x^{m}_i < x^{M}_i$, then $x^{m}_i < u^{M}_i$. Hence $u^{m}_i < x^{m}_i < u^{M}_i$. 
Let $(u,v)$ and $(u',v')$ be the two pairs of words corresponding respectively with $c$ and $c'$.
The $(n-1)$-tuple $c'$ is a cubic coordinate. Indeed, since $v=v'$ and $c$ is a cubic coordinate, $v'$ is a dual Tamari diagram. Therefore, since $e$ is a cubic coordinate, $c'_k = e_k$ if $c'_k$ is positive for all $i\leq k\leq n-1$, and one has that $\uparrow_i (c)$ is also a cubic coordinate, then $u'$ is a Tamari diagram. Finally, since $\uparrow_i (c)\in\cc$, we can conclude that Definition~\ref{compatibleDT} is satisfied and $c'$ is a cubic coordinate. 
It leads us to the fact that there is a cubic coordinate $c'$ distinct from $c$ and $\uparrow_i (c)$ such that $c \leq_{\mathrm{cc}} c' \leq_{\mathrm{cc}} \uparrow_i (c)$, which is impossible by Definition~\ref{augMin}. 
Whence $c^{m}_i = e^{m}_i$.

\item Let us suppose that~\ref{injGam2} holds. We show that $c^{m}_i = e^{m}_i$ and $c^{M}_i = e^{M}_i$ in the same way as~\ref{injGam1}, by reformulating the arguments for the dual case. 
\end{itemize}

By definition of $\gamma$, we cannot have any other case. Therefore $\Gamma$ is injective.

Now, let us show that the cardinality of the set of cells of size $n$ is equal to the cardinality of $\ccs$. Recall that the set of cells of size $n$ is exactly the set of minimal-cellular cubic coordinates of size $n$. Moreover, this is also the set of cubic coordinates which are covered by exactly $n-1$ elements in $(\cc, \leq_{\mathrm{cc}})$. Besides, by the isomorphism of posets $\psi$, we know that these elements correspond to the ones with the same property in the poset of Tamari intervals.
In~\cite{Cha17}, Chapoton shows that the set of these elements have the same cardinality as the set of synchronized Tamari intervals (see Theorem 2.1 and Theorem 2.3 there). By Proposition~\ref{ditsyn=tisyn}, it may be inferred that the cardinality of $\ccs$ and the cardinality of the set of cells of size $n$ are equal.

One can conclude that $\Gamma$ is bijective.
\end{proof}

%%%%%%%%%%%%%%%%%%%%%%%%%%%%%%%%%%%%%%%%%%%%%%%%%%%%%%%%%%%

%\acknowledgements{}

%% if you use biblatex then this generates the bibliography
%% if you use some other method then remove this and do it your own way
%\printbibliography
\bibliographystyle{plain}
\bibliography{Bibliography}

\begin{thebibliography}{10}

\bibitem{BPR12}
F.~Bergeron and L.-F. Pr\'{e}ville-Ratelle.
\newblock Higher trivariate diagonal harmonics via generalized {T}amari posets.
\newblock {\em J. Combin.}, (3):317--341, 2012.

\bibitem{BB09}
O.~Bernardi and N.~Bonichon.
\newblock Intervals in {C}atalan lattices and realizers of triangulations.
\newblock {\em J. Combin. Theory Ser. A}, 116(1):55--75, 2009.

\bibitem{BMFPR12}
M.~Bousquet-M\'elou, \'E. Fusy, and L.-F. Pr\'{e}ville-Ratelle.
\newblock The number of intervals in the $m$-{T}amari lattices.
\newblock {\em Electronic J. Combin.}, 18(2), 2012.

\bibitem{Cha06}
F.~Chapoton.
\newblock Sur le nombre d'intervalles dans les treillis de {T}amari.
\newblock {\em S\'{e}m. Lothar. Combin.}, 55:Art. B55f, 18, 2006.

\bibitem{Cha17}
F.~Chapoton.
\newblock Une note sur les intervalles de {T}amari.
\newblock {\em Ann. Math. Blaise Pascal}, 25(2):299--314, 2018.

\bibitem{CP15}
G.~Châtel and V.~Pons.
\newblock {Counting smaller elements in the Tamari and $m$-Tamari lattices}.
\newblock {\em {J. Combin. Theory Ser. A}}, 134: 58–97, 2015.

\bibitem{FPR17}
W.~Fang and L.-F. Pr\'{e}ville-Ratelle.
\newblock The enumeration of generalized {T}amari intervals.
\newblock {\em European J. Combin.}, 61:69--84, 2017.

\bibitem{Gir11}
S.~Giraudo.
\newblock {Combinatoire algébrique des arbres}.
\newblock {\em {PhD thesis}}, 2011.
\newblock Section 4.4.3.

\bibitem{Pal86}
J.~M. Pallo.
\newblock {Enumerating, ranking and unranking binary trees}.
\newblock {\em {Comput. J. 29}}, no. 2, 171–175, 1986.

\bibitem{PRV17}
L.-F. Pr\'eville-Ratelle and X.~Viennot.
\newblock The enumeration of generalized {T}amari intervals.
\newblock {\em Trans. Amer. Math. Soc.}, 369(7):5219--5239, 2017.

\bibitem{Rog18}
B.~Rognerud.
\newblock {Exceptional and modern intervals of the Tamari lattice}.
\newblock {\em To appear in S\'em. Lothar. Combin.}, 2019.

\bibitem{Tam62}
D.~Tamari.
\newblock The algebra of bracketings and their enumeration.
\newblock {\em Nieuw Arch. Wisk. (3)}, 10:131--146, 1962.

\bibitem{Tro02}
W.~T. Trotter.
\newblock {\em Combinatorics and partially ordered sets: {D}imension theory}.
\newblock Johns Hopkins Series in the Mathematical Sciences. The Johns Hopkins
  University Press, 2002.

\end{thebibliography}

\end{document}